\documentclass[12pt]{amsart}
\usepackage{eurosym}
\usepackage{latexsym,amssymb,amsmath,bbm,amsfonts,amsthm,
mathrsfs,bbm,euscript}

\newtheorem{theorem}{Theorem}[section]

\newtheorem{proposition}[theorem]{Proposition}
\newtheorem{corollary}[theorem]{Corollary}

\newtheoremstyle{notauto}{}{}{\itshape}{}{\bfseries}{.}{0.5em}{\thmnote{#3}}
\theoremstyle{notauto}

\theoremstyle{definition}
\newtheorem{definition}[theorem]{Definition}
\newtheorem{example}[theorem]{Example}
\theoremstyle{remark}
\newtheorem{remark}[theorem]{Remark}

\renewcommand{\leq}{\leqslant}

\begin{document}
\title{good action on a finite group}
\subjclass[2000]{20D10, 20D15, 20D45}
\keywords{solvable group, automorphism, Fitting height, fixed point free
action, regular module}
\author{G\"{u}l\. In Ercan$^{*}$}
\address{G\"{u}l\. In Ercan, Department of Mathematics, Middle East
Technical University, Ankara, Turkey}
\email{ercan@metu.edu.tr}
\thanks{$^{*}$Corresponding author}
\author{\.{I}sma\. Il \c{S}. G\"{u}lo\u{g}lu}
\address{\.{I}sma\. Il \c{S}. G\"{u}lo\u{g}lu, Department of Mathematics, Do%
\u{g}u\c{s} University, Istanbul, Turkey}
\email{iguloglu@dogus.edu.tr}
\author{Enrico Jabara}
\address{Enrico Jabara, Dipartimento di Filosofia, Universit\`a Ca' Foscari
di Venezia, Venice, Italy}
\email{jabara@unive.it}

\begin{abstract}
Let $G$ and $A$ be finite groups with $A$ acting on $G$ by automorphisms. In
this paper we introduce the concept of ``good action"; namely we say the
action of $A$ on $G$ is good, if $H=[H,B]C_H(B)$ for every subgroup $B$ of $A
$ and every $B$-invariant subgroup $H$ of $G.$ This definition allows us to
prove a new noncoprime Hall-Higman type theorem.

If $A$ is a nilpotent group acting on the finite solvable group $G$ with $%
C_G(A)=1$, a long standing conjecture states that $h(G)\leq \ell(A)$ where $%
h(G)$ is the Fitting height of $G$ and $\ell(A)$ is the number of primes
dividing the order of $A$ counted with multiplicities. As an application of
our result we prove the main theorem of this paper which states that the
above conjecture is true if $A$ and $G$ have odd order, the action of $A$ on 
$G$ is good and some other fairly general conditions are satisfied.
\end{abstract}

\maketitle

\section{Introduction}

Let $A$ be a finite group that acts on the finite group $G$ by
automorphisms. We write $h(G)$ for the Fitting height of $G$, and $\ell (A)$
for the length of the longest chain of subgroups of $A$ which coincides with
the number of primes dividing the order of $A$ counted with multiplicities
if $A$ is solvable. Thompson \cite{Tho} proved that in case where $A$ and $G$
are both solvable and $(|G|,|A|)=1,\,h(G)$ is bounded in terms of $%
h(C_{G}(A))$ and $\ell (A)$. Thompson's result has inspired the work of many
authors and has been refined in particular in \cite{K}, \cite{HaIs}, and 
\cite{Tur}. Namely, in \cite{Tur} Turull obtained that $h(G)\leq
h(C_{G}(A))+2\ell (A).$

Due to the lack of some nice consequences of coprime action the situation is
very difficult to handle without the coprimeness condition $(|G|,|A|)=1$. An
example obtained by Bell and Hartley \cite{Bell} shows that for any
nonnilpotent finite group $A$, there exists a finite group $G$ of
arbitrarily large Fitting height on which $A$ acts fixed point freely and
noncoprimely. However if $A$ is nilpotent and $C_G(A) = 1$, a special case
of Dade's theorem \cite{Dade} provides an exponential bound for $h(G)$ in
terms of $\ell(A)$. Some improvements of this bound are obtained in
particular cases, e.g. see \cite{Ja} for cyclic $A$. But apparently,
improving to a linear bound is a difficult problem.

A celebrated work of Thompson \cite{Tho2} asserts that every finite group
admitting a fixed point free automorphism of prime order is nilpotent. This
result gave birth to the long standing conjecture that, under the
coprimeness condition $(|G|, |A|) = 1$, if $C_G(A)=1$ then $h(G)$ is at most 
$\ell (A)$. There have been a great amount of work on this problem for
various cases of $A$ and finally Turull settled the conjecture for almost
all $A$ in a sequence of papers (see in particular \cite{Turfix}, \cite{Tur}
and \cite{Tursur}) by proving the following result.

\vspace{2mm} \textbf{Theorem} \textit{Let $A$ be a finite group acting by
automorphisms on the finite solvable group $G$ such that $(|G|,|A|) = 1$ and 
$C_G(A)=1$. If every proper subgroup of $A$ acts with regular orbits on $G$
then $h(G)\leq \ell(A).$} \vspace{2mm}

Here a group $B$ is said to act with regular orbits on another group $G$ if
for any $B$-invariant section $S$ of $G$ on which $B$ acts irreducibly there
exists $x\in S$ such that $C_B(x)=C_B(S)$ so that the $B$-orbit (which is
actually an $B/C_B(S)$-orbit) is a regular orbit, that is an orbit of length 
$|B/C_B(S)|.$ It should be noted that there are large classes of finite
groups $A$ always acting with regular orbits on any finite group $G$ of
coprime order on which it acts. But there also exist many finite groups
which do not need to act with regular orbits.

The example due to Bell and Hartley \cite{Bell} mentioned above forced to
state the noncoprime version of the conjecture as follows: \vspace{2mm}

\textbf{Conjecture} \textit{If $A$ is a finite nilpotent group acting fixed
point freely on a finite (solvable) group $G$ by automorphisms then $h(G)
\leq \ell (A)$. } \vspace{2mm}

Although the noncoprime version has been proven in some special cases (\cite%
{cheng}, \cite{EGJGT},\cite{EG},\cite{EGS},\cite{EMo}), it is still
unproven even in the case where $A$ is cyclic.

In the present paper we introduce the concept of a good action of $A$ on $G$%
; namely we say the action is ``\textit{good}" if $H=[H,B]C_H(B)$ and for
every subgroup $B$ of $A$ and for every $B$-invariant subgroup $H$ of $G$.
It can be regarded as a generalization of the coprime action due to the fact
that every coprime action is good. Some other features of the coprime
action, e.g. the existence of an $A$-invariant Hall subgroups, are actually
consequences of the fact that coprime action is a good action as we show in
Proposition 2.2. On the other hand there are noncoprime actions which are
good (see Remark 2.4). So it is natural to ask whether one can get results
in the noncoprime case which are similar to the above theorem due to Turull
if some nice consequences of a coprime action are kept. In other words one
can ask whether the relative easiness of the proofs in the coprime case is
due to ``goodness" of the action or not. Our main result is a partial answer
which provides the best possible upper bound for the Fitting height of a
solvable group of odd order admitting a good and fixed point free action.
Namely we prove the following. \vspace{2mm}

\textbf{Theorem} (Theorem 4.5)\textit{Let $A$ be a finite nilpotent group of
odd order which is $C_q\wr C_q$-free for every prime $q$ act on the finite group $G$. Suppose that
\begin{itemize}
\item[(a)] $C_G(A) = 1$;
\item[(b)] the action of $A$ on $G$ is good;
\item[(c)] every subgroup of $A$ acts on $G$ with regular orbits;
\item[(d)] there is a subgroup $B$ of $A$ such that $\bigcap_{a\in A} \left[ G,B \right]^a =1$.
\end{itemize} 
Then $G$ is a solvable group of Fitting
height at most $\ell(A:B)$ where $\ell(A:B)$ is the number of prime divisors
of $|A:B|$ counted with multiplicities. } \vspace{2mm}

Notice that we always have $\bigcap_{a\in A} \left[ G,B \right]^a =1$ for $%
B=1$ which yields the following. \vspace{2mm}

\textbf{Corollary} \textit{Let $A$ be a finite nilpotent group of odd order
which is $C_q\wr C_q$-free for every prime $q$ act on the finite group $G$. Suppose that this action is good and every subgroup of $A$ acts with
regular orbits on $G$. If $C_G(A)=1$ then $G$ is a solvable group of Fitting height at most $%
\ell(A)$.} \vspace{2mm}

The proof of this theorem follows the model of the proof of Theorem 2.2 in 
\cite{Turfix}. Although at some points we could have avoided the details by
adding ``by an argument similar to ..." and referring to \cite{Turfix}; for
the reader's convenience, we have formulated them so that they can be easily
followed without any further references.

It should also be noted that in order to overcome some of the main
difficulties arising from noncoprimeness we needed a new Hall-Higman type
theorem. In this direction we obtained Theorem 3.1 and Theorem 3.3 which are
of independent interest in the study of noncoprime action problems.

Throughout the paper all groups are finite, notation and terminology are
standard.

\section{Good Action}

In this section we introduce the concept of a good action. Some of its
immediate consequences are obtained below.

\begin{definition}
\textit{Let $G$ and $A$ be groups where $A$ acts on $G$ by automorphisms. We
say ``\textit{the action of $A$ on $G$ is good}" if the equality $%
H=[H,B]C_H(B)$ holds for any subgroup $B$ of $A$ and for any $B$-invariant
subgroup $H$ of $G$.}
\end{definition}

\begin{proposition}
Let $G$ and $A$ be groups where $A$ acts on $G$ by automorphisms, and
suppose that the action of $A$ on $G$ is good. If $B\leq A$ and $N$ is a
normal $B$-invariant subgroup of $G$ then

\begin{itemize}
\item[(1)] the action of $B$ on every $B$-invariant subgroup of $G$ is good;

\item[(2)] the equality $[H,B,B]=[H,B]$ holds for any $B$-invariant subgroup 
$H$ of $G$;

\item[(3)] the equality $C_{H/N}(B)=C_H(B)N/N$ holds for any $B$-invariant
subgroup $H$ of $G$ containing $N$;

\item[(4)] the induced action of $B$ on $G/N$ is good.
\end{itemize}
\end{proposition}

\begin{proof}
(1) and (2) are straightforward. To prove (3) set $X/N=C_{H/N}(B)$. Since $X$
is $B$-invariant, we have $X=[X,B]C_X(B)$ with $[X,B]\leq N$. It follows
that $X=NC_X(B)$ whence $C_H(B)N/N=X/N$. Notice next that $H=[H,B]C_H(B)$
implies 
\begin{equation*}
H/N=[H,B]C_H(B)N/N=[H,B]N/NC_{H/N}(B)=[H/N,B]C_{H/N}(B) 
\end{equation*}
as $C_H(B)N/N=C_{H/N}(B)$ proving (4).
\end{proof}

\begin{proposition}

Let $G$ and $A$ be solvable groups where $A$ acts on G by automorphisms. The
action of $A$ on $G$ is good whenever $(|[G,A]|,|GA/[G,A]|)=1$.
\end{proposition}

\begin{proof}
Let $B\leq A$. Then for any $B$-invariant subgroup $H$ of G, $H\cap
[G,A]=H_0\lhd H$ and $H/H_0\cong [G,A]H/[G,A]$ has order coprime to $H_0$.
Therefore $HB$ has a Hall subgroup $T$ containing $B$ which is a complement
of $H_0$ in $HB$. Now $H=H_0(T\cap H)$ and clearly we have $%
H_0=[H_0,B]C_{H_0}(B)$ as $(|H_0|,|B|)=1$. We also have $[T\cap H,B]\leq
T\cap H_0=1$, yielding that $H=[H,B]C_{H}(B)$.
\end{proof}

\begin{remark}
Motivated by Proposition 2.3 we can construct some concrete examples of good
action as follows:\newline

\textit{Example 1.} Let $H=(C_5\times C_5)\rtimes SL(2,3)$ be the Frobenius
group with kernel isomorphic to $C_5\times C_5$. Since $H$ has a unique
minimal normal subgroup and $O_2(H)=1$ there exists a faithful
irreducible $H$-module $V$ over $GF(2).$ Let $G=V\rtimes H$ and $\langle
x\rangle$ be a Sylow $3$-subgroup of $G$. If $\alpha$ is the inner
automorphism induced by $x$ on $G$ then the action of $A=\langle \alpha
\rangle$ on $G$ is good by the above proposition.\newline

A slight modification of Example 1 shows that $A$ need not be a subgroup of $Inn(G)$:\newline

\textit{Example 2.} Let $G=R\langle x\rangle$ be the group built in Example
1 where $R=O_{3^{\prime}}(G)$, and let $T=\langle t\rangle\cong C_2$. Set $%
\Gamma=G\wr T$ and let $\Gamma_0\cong (R_1\times R_2)\rtimes S_3$ where $%
R_1\cong R\cong R_2$ with ${R_1}^t=R_2$ and $S_3$ is the symmetric group of
degree $3$ generated by $\{x^{-1}x^t,t\}$. If $\beta$ is the automorphism
induced by conjugation by $xx^t$ on $\Gamma_0$, then one can verify that the
action of $A=\langle \beta \rangle$ on $\Gamma_0$ is good. Moreover $\beta$
is not an inner automorphism of $\Gamma_0.$ 
\end{remark}

\begin{proposition}
Let $G$ and $A$ be groups where $A$ acts on G by automorphisms, and suppose
that the action of $A$ on $G$ is good. Let $p\in \pi(A)$ and let $B$ be a $p$%
-subgroup of $A$. If $G$ is $p$-solvable then $[G,B]$ is a $p^{\prime}$%
-group.
\end{proposition}

\begin{proof}
Set $H=[G,B]$ and $\overline{H}=H/O_{p^{\prime }}(H)$. By Proposition 2.2
the action of $B$ on $\overline{H}$ and hence on $X=O_p(\overline{H})$ is
good. Then $[X,B,B]=[X,B]$ by Proposition 2.2 and so $[X,B]=1$. As $[B, X, 
\overline{H}]=1=[X, \overline{H},B]$, by the three subgroups lemma we have $%
[ \overline{H} ,B,X]=1$. Since $G$ is $p$-solvable we get $[\overline{H}%
,B]\leq C_{\overline{H}}(X)\leq X$. Then $[\overline{H},B]=[\overline{H}%
,B,B]\leq [X,B]=1$ by Proposition 2.2. This establishes the claim.
\end{proof}

\begin{proposition}
Let $G$ be a solvable group acted on by a nilpotent group $A$. If the action
is good, then there exists an $A$-invariant Hall $\sigma$-subgroup in
$G$ for any nonempty $\sigma\subseteq\pi(G)$.
\end{proposition}

\begin{proof}
We proceed by induction on $|G|+|\pi(A)|.$ Let $N$ be a minimal normal subgroup of $GA$ contained in $G.$ Then $A$ acts
on $G/N$ by automorphisms and this action is good. It follows by induction that $G/N$ contains an
$A$-invariant Hall $\sigma$-subgroup, say $X/N.$ If $X\neq G$ then $A$ acts on $X$
and this action is also good. So $X$ contains an $A$-invariant Hall $\sigma
$-subgroup $Y$ by induction. Clearly $Y$ is an $A$-invariant Hall $\sigma$-subgroup of
$G.$ Therefore we can assume that $G/N$ is a $\sigma$-group. Note also that $N$ is an
elementary abelian $r$-group for some prime $r$ and we can assume that
$r\notin\sigma.$

Let now $K$ be a normal subgroup of $GA$ such that $N\leq K<G$ and $G/K$ is an elementary abelian $s$-group for some prime $s$. As $K$ is
$A$-invariant, it contains an $A$-invariant Hall
$\sigma$-subgroup, say $S,$ by induction. Then we have $K=NS$, and $G=NN_{G}(S)$ by the Frattini argument. Here $N_{G}(S)$ is $A$-invariant. If $N\cap N_{G}(S)=1$ then $N_{G}(S)$ is an $A$-invariant Hall $\sigma$-subgroup of $G.$ If $N\cap
N_{G}(S)\neq 1$ we have $N\leq N_{G}(S)$ since $N$ is an irreducible $N_{G}(S)A$-module, and hence $S\trianglelefteq G.$ When $S\neq1$ we apply induction to $G/S$ to get the claim. Thus we may assume that
$S=1$, whence $K=N$, $G$ is an $\{r,s\}$-group, $\sigma=\{s\}$ and
$O_{s}(G)=1.$

If $(|G|,|A|)=1,$ then the result is
well known. Suppose first that $r\in\pi(A)$ and let $A_r$ denote the Sylow $r$-subgroup of $A$. Because the action of $A_r$ on $G$ is also good, we have $[N,A_r]=1$ by Proposition 2.5. This implies by the
three subgroup lemma that $[[G,A_r],N]=1$ which shows that
$[G,A_r]\trianglelefteq C_{G}(N)=N$ and hence $[G,A_r]=[[G,A_r],A_r]=1.$ So we can
assume that $r\notin\pi(A)$ and $s\in\pi(A).$ Let $A_s$ be the Sylow
$s$-subgroup of $A$ and $T$ a Sylow $s$-subgroup of $GA$ containing $A_s$. Put $U=T\cap G$. Then $U$ is a Sylow $s$-subgroup of $G=NU$ and is normalized by
$A_s.$ As the action of $A_s$ on $U$ is good we get that $U\leqq C_{G}(A_s).$ As
$C_{N}(A_s)$ is $GA$-invariant and cannot be equal to $N$ we obtain
$C_{N}(A_s)=1$ and hence  $C_{G}(A_s)=U.$ So $U$ is an $A$-invariant Hall $\sigma
$-subgroup of $G.$ This completes the proof.
\end{proof}

In the rest of this section we study the relation between the Fitting
heights of $G$ and $C_G(A)$ in case of a good action.

\begin{definition}
(Definition $1.1$ and $1.2$ of \cite{Tur}) Let $G$ and $A$ be finite groups
where $A$ acts on $G$. We say that a sequence $(S_i),i= 1,\ldots ,h$ of $A$%
-invariant subgroups of $G$ is an $A$-tower of $G$ of height $h$ if the
following are satisfied:

\begin{itemize}
\item[(1)] $S_i$ is a $p_i$-group, $p_i$ is a prime, for $i= 1,\ldots ,h$;

\item[(2)] $S_i$ normalizes $S_j$ for $i\leq j$;

\item[(3)] Set $P_h=S_h,\, P_i=S_i/C_{S_i}(P_{i+1}), \, i= 1,\ldots ,h-1$
and we assume that $P_i$ is not trivial for $i= 1,\ldots ,h$;

\item[(4)] $p_i\ne p_{i+1}, \, i= 1,\ldots ,h-1$.
\end{itemize}

\vspace{1mm}

An $A$-tower $(S_i),\, i=1,\ldots ,h$ of $G$ is said to be irreducible if
the following are satisfied:

\begin{itemize}
\item[(5)] $\Phi(\Phi(P_i)) = 1,\, \Phi(P_i)\leq Z(P_i)$ and, if $p_i\ne 2$,
then $P_i$ has exponent $p_i$ for $i= 1,\ldots ,h.$ Moreover $P_{i-1}$
centralizes $\Phi(P_i)$;

\item[(6)] $P_1$ is elementary abelian;

\item[(7)] There exists $H_i$ an elementary abelian $A$-invariant subgroup
of $P_{i-1}$ such that $[H_i,P_i] = P_i$ for $i=2,\ldots ,h$;

\item[(8)] $(\prod_{j=1}^{i-1}S_j)A$ acts irreducibly on $P_i/\Phi(P_i)$.
\end{itemize}
\end{definition}

\begin{remark}
Moreover we can easily verify that Lemma 1.9 of \cite{Tur} requires that $G$
contains an $A$-invariant Sylow $p$-subgroup for every prime $p$ dividing $%
|G|$. Hence our Proposition 2.6 and Lemma 1.9 of \cite{Tur} yields 
\begin{equation*}
h(G) = \max \big \{h\; : \; \mbox{there is an $A$-tower of height $h$ in $G$}
\big \}. 
\end{equation*}
\end{remark}

\begin{theorem}
Let $A$ be a group of prime order $p$ acting on the solvable group $G$ such
that $C_G(A)$ is of odd order. If the action of $A$ on $G$ is good, then $%
h(G)\leq h(C_G(A))+4.$
\end{theorem}

\begin{proof}
Set $h = h(G)$. By the above remark there exists an irreducible $A$-tower $%
(S_i)$ of height $h$ in $G$, with $P_i = S_i/T_i, i=1,\ldots ,h$ and $P_0=1.$
Let $k$ be the largest integer with the property that $[P_k,A] = 1$. As the
action is good, by Proposition 2.2 (3) we have $C_{P_i}(A) =
C_{S_i}(A)T_i/T_i$ for each $i$, and hence we may assume that $S_i$ is
centralized by $A$ for each $i = 1,\ldots ,k$ and that $[S_{k+1},A] = S_{k+1}
$. It follows by Proposition 2.5 that $p_i\ne p$ for each $i = k +1,\ldots, h
$, that is $\prod_{i=k+1}^h S_i$ is a $p^{\prime}$-group. Notice also that $%
h-k > 4$ because $k\leq h(C_G(A))$.

Set $C_i = C_{S_i}(A)$ for each $i.$ Then $C_i = S_i$ for $i = 1,\ldots ,k$.
Also note that $p_k$ is odd as $|C_G(A)|$ is odd.

Suppose first that $p_k\ne p$, by Theorem 3.1 in \cite{Tur} applied to $%
S_h,\ldots , S_k$, we obtain $j\in \{k + 1, \ldots ,h\}$ such that the
sequence $C_h,\ldots ,C_{j+1},C_{j-1},\ldots ,C_k$ satisfies the conditions
(1), (2) and (3) of Definition 2.7, possibly with $p_{j+1} = p_{j-1}$. It
follows that either the sequence 
\begin{equation*}
C_h,\ldots , C_{j+1}, C_{j-1},\ldots , C_k,\ldots , C_1
\end{equation*}
forms an $A$-tower of height $h-1$ or the sequence 
\begin{equation*}
C_h,\ldots , C_{j+2}, C_{j-1},\ldots ,C_k,\ldots ,C_1
\end{equation*}
forms an $A$-tower of height $h-2$, and hence $h(G)\leq h(C_G(A)) + 2$.

Suppose next that $p_k = p$. Note that if $k = 1$ then $S_h,\ldots ,S_2$
forms an $A$-tower whose terms are $p^{\prime}$-subgroups, and by the main
result of \cite{Tur}, we get $h-1=h(\prod_{i=2}^h S_i)\leq h(C_G(A)) + 2$.
Therefore we may assume that $k > 1$. Notice that we have either $p_{k-1}$
and $p_{k+1}$ are equal, or not. If the former holds then the sequence $%
S_h,\ldots ,S_{k+2}, S_{k-1}$ forms an $A$-tower. As $p_{k-1}$ is odd, we
apply Theorem 3.1 in \cite{Tur}, and obtain $j\in \{k + 2,\ldots ,h\}$ such
that $C_h,\ldots ,C_{j+1}, C_{j-1},\ldots ,C_{k-1}$ is a sequence satisfying
the conditions (1), (2) and (3) of Definition 2.7, possibly with $p_{j+1} =
p_{j-1}$. It follows that either the sequence 
\begin{equation*}
C_h,\ldots ,C_{j+1}, C_{j-1},\ldots , C_{k+2}, C_{k-1},\ldots ,C_1
\end{equation*}
forms an $A$-tower of height $h-3$ or the sequence 
\begin{equation*}
C_h,\ldots ,C_{j+2}, C_{j-1},\ldots , C_{k+2}, C_{k-1},\ldots ,C_1
\end{equation*}
forms an $A$-tower of height $h-4$. Then we have $h\leq h(C_G(A)) + 4$ and
the theorem follows. Finally suppose that $p_{k-1}\ne p_{k+1}$. Now $%
S_h,\ldots ,S_{k+1}, S_{k-1}$ forms an $A$-tower. Notice again that as $%
p_{k-1}$ is odd Theorem 3.1 in \cite{Tur} gives $j\in \{k + 1,\ldots ,h\}$
such that the sequence $C_h,\ldots ,C_{j+1}, C_{j-1},\ldots , C_{k+1},
C_{k-1}$ satisfies the conditions (1), (2) and (3) of
Definition 2.7. It follows that either the sequence 
\begin{equation*}
C_h,\ldots ,C_{j+1}, C_{j-1},\ldots , C_{k+1}, C_{k-1}\ldots ,C_1
\end{equation*}
forms an $A$-tower of height $h-2$ or the sequence 
\begin{equation*}
C_h,\ldots ,C_{j+2}, C_{j-1},\ldots , C_{k+1}, C_{k-1}\ldots ,C_1
\end{equation*}
forms an $A$-tower of height $h-3$. Then we have $h(G)\leq h(C_G(A)) + 3$
and this completes the proof.
\end{proof}

\begin{theorem}
Let $A$ be a solvable group acting on the solvable group $G$ such that $%
C_G(B)$ is of odd order for every nontrivial subgroup $B$ of $A$. If the action is good then $h(G)\leq h(C_G(A)) +
4\ell(A)$.
\end{theorem}

\begin{proof}
Let $G$ be a minimal counterexample to the theorem and let $B\lhd A$ such
that $A/B$ is of prime order. By induction applied to the action of $B$ on $G
$ we see that $h(G)\leq h(C_G(B)) + 4\ell(B)$. Theorem 2.9 applied to the
action of $A/B$ on $C_G(B)$ yields that $h(C_G(B))\leq h(C_G(A)) + 4$,
whence $h(G)\leq h(C_G(A)) + 4\ell(A)$.
\end{proof}

As an immediate consequence we have the following.

\begin{corollary} Let $A$ be a solvable group acting on the solvable group $G$ of odd order. If the action is good then $h(G)\leq h(C_G(A)) +
	4\ell(A)$.
\end{corollary}

\section{A noncoprime Hall-Higman Type Theorem}

This section is devoted to the study of some technical problems pertaining
to the proof of our main theorem. The following results are also of
independent interest because they seem to be effectively applicable in other
situations of noncoprime action.

\begin{theorem}
Let $A$ be a nilpotent group acting with regular orbits on the group $G$.
Let V be a complex $GA$-module so that $V_G$ is homogeneous on which $A$
acts fixed point freely. Suppose that $G/N$ is a $GA$-chief factor of $G$
which is an elementary abelian $r$-group for some prime $r$. If $A$
normalizes a Hall $r^{\prime}$-subgroup of $GA,$ then there exists a
homogeneous component $U$ of $V_{_{N}}$ and a subgroup $B$ of $A$ such that $%
B\leq N_A(U)$, $C_V(B)=0$ and $[G,B]\leq N_G(U)$. In particular we have $%
[G,B]\leq N$ in case where $V_{N}$ is not homogeneous.
\end{theorem}

\vspace{1mm} \noindent \textit{Proof.} It is useful to proceed in a series
of steps.\newline

\textit{(1) We may assume that $V_{_{N}}=W_1\oplus\ldots \oplus W_s $ with $s>1$ where $%
W_i = m Y_i$, $i=1,\ldots ,s$, for some positive integer $m$ and for
pairwise nonisomorphic irreducible $N$-submodules $Y_1,\ldots , Y_s$.
Furthermore, $G$ acts transitively and $G/N$ acts regularly on $%
\Omega=\{W_1,\ldots ,W_s \}$.}

\begin{proof}
Let $X$ be an irreducible submodule of $V_{_{G}}$. Then $X_{_{N}}$ is
irreducible or homogeneous or a sum of nonisomorphic irreducible submodules
by \cite{isaacs}, 6.18. In the former and the second cases, $V_{_{N}}$ is
homogeneous and we take $V=U$ and $B=A$. We may therefore assume that the
latter holds, that is, $X_{_{N}}=Y_1\oplus \cdots \oplus Y_s $, where $Y_i$
are pairwise nonisomorphic irreducible $N$-submodules. Then $V_{_{N}}= m
Y_1\oplus \cdots \oplus m Y_s $ for some positive integer $m$. Set $W_i=mY_i$%
, $i=1,\ldots,s$. As $V_G$ is homogeneous, we may assume that $G$ acts
transitively on $\Omega=\{W_1,\ldots ,W_s \}$. Hence $GA= N_{GA}(W_i)G$ for
each $i=1,\ldots ,s$. We may also observe that $N=N_G(W_i)$ for each $%
i=1,\ldots ,s$, as $G/N$ is a $GA$-chief factor, that is, $G/N$ acts
regulary on $\Omega$.
\end{proof}

\textit{(2) $G/N$ is an elementary abelian $r$-group for some prime $r$. Set 
$A=A_r\times A_{r^{\prime}}$. Then $A_r$ centralizes $G/N$.}

\begin{proof}
This follows from the irreducibility of $G/N$ as a $GA$-module.
\end{proof}

\textit{(3) Set $M=N_{GA}(W_1)$. Then $G/N$ is centralized by $K=Core_{GA}(M)
$ and hence $K=C_{M}(G/N)$. Note that $N=K \cap G$. }

\begin{proof}
Now $N\leq M$. Notice that $MG=GA$ as $G$ acts transitively on $\Omega$ by 
\textit{(1)}. Then $G\not \leq M.$ We have $N\leq K \cap G \triangleleft GA$
and hence $K \cap G=N$ by the irreducibility of $G/N$ as a $GA$-module. It
follows that $[G,K] \leq N$, that is, $K$ centralizes $G/N$. Then $K \leq L
=C_{M}(G/N)$. Now, $L$ is normalized by $G$ as $[G,L] \leq N\leq L$. $L$ is
also left invariant by $M$. Therefore $L$ is $GA$ -invariant as $GA=MG$.
This shows that $L \leq \bigcap _{x \in GA} M^x =K$ and hence we have the
equality $L=K$, as desired.
\end{proof}

\textit{(4) Set $\overline{GA}= GA/K$. Then $G/N \cong \overline{G}= O_r( 
\overline{GA})$. Furthermore $O_{r,r^{^{\prime}}}( \overline{GA}) = 
\overline{GA}$. }

\begin{proof}
By \textit{(3)}, $N =K \cap G$ and hence $G/N$ is an $r$-group. Then $%
\overline{G}\leq O_r (\overline{GA})$. On the other hand, 
\begin{equation*}
GA/KG=KGA/KG\cong A/A\cap KG
\end{equation*}
and $A_r(A \cap KG) /(A\cap KG)$ is trivial, because 
\begin{equation*}
A_r \leq C_{GA}(G/N)=C_{MG}(G/N)=GC_M(G/N)=KG 
\end{equation*}
by \textit{(3)}. This yields that $\left|GA/KG\right|$ is not divisible by $r
$ and hence $\overline{G}= O_r ( \overline{GA})$. In fact, $%
O_{r,r^{^{\prime}}}(\overline{GA}) = \overline{ GA}$.
\end{proof}

\textit{(5) $A\leq M^{x_0}=M_0$ for some $x_0\in GA$.}

\begin{proof}
Let now $Q$ be a Hall $r^{\prime}$-subgroup of $GA$ normalized by $A$. Then $%
\overline{Q}$ is an $A$-invariant Hall $r^{\prime}$-subgroup of $\overline{GA%
}$ and hence $\overline{Q} O_r(\overline{GA})= \overline{Q} \overline{G}=%
\overline{GA}$ by \textit{(4)}. On the other hand, we also have $\overline{GA%
}= \overline{M}\overline{G}$. Thus $\overline{Q}$ and $\overline{M}$ are conjugate in $\overline{GA}$,
that is, $\overline{Q}=\overline{M}^{ \overline {x_0}}$ for some $\overline{%
x_0}=x_0K$ in $\overline{GA}$. It follows that $M^{x_0}=QK$. Set $M_0=M^{x_0}
$ and $W_0={W_1}^{x_0}.$

Since $\overline{GA}=\overline{Q} \overline{G}$ we see that $N_{\overline{GA}}(\overline{Q})=N_{\overline{G}}(\overline{Q})\overline{Q}=C_{\overline{G}}(\overline{Q})\overline{Q}$. Notice that $C_{\overline{G}}(\overline{Q})=1$ or $\overline{G}$ by the irreducibility of $\overline{G}$ as an $A$-module. In the former case $A\leq N_{GA}(Q)K=QK=M_0$ and hence the claim follows. In the latter case we have $M\unlhd GA$. Hence $M_0=K$ and $\overline{GA}=\overline{G}$ which yields that $\overline{A}=1$. Therefore in any case we have $A\leq M_0$ as claimed. 
\end{proof}

\textit{(6) Theorem follows.}

\begin{proof}
We have $V={W_0}^{GA}$ where $W_0$ is an $M_0$-module. Choose a
subset $\Re$ of $G$ which consists of double coset representatives of $%
(M_0,A)$ in $G$. We observe that $C_A(xN)=A \cap {M_0}^{x}$ for each $x \in
\Re$: To see this, let $b \in C_A(xN)$. Then $b \in {M_0}^{x}$ as $A\leq M_0.
$ Conversely, pick an element $b$ from $A\cap {M_0}^{x}$. Then $[b,
x^{-1}]\in N$, that is, $[b,xN]=1$ which yields the desired equality.

Recall that $A$ acts with regular orbits on $G/N$ by hypothesis. More
precisely, there exists $x_1 \in \Re$ such that $C_A(x_1N)=C_A(G/N)$. Note
that 
\begin{equation*}
V_{_{A}}=\bigoplus_{x \in R} {{W_0}^{x}}|_{_{A \cap {M_0}^{x}}}|^{A}=%
\bigoplus _{x \in R} W_0^{x}|_{_{C_A(xN)}}|^A
\end{equation*}
by Mackey's theorem. Taking now $U={W_0}^{x_1}$ and $B=C_A(G/N)=A\cap {M_0}%
^{x_1}$, we see that $U_{_{B}}|^{A} \subseteq V_{_{A}}$ and hence $C_V(A)=0$
implies $C_U(B)=0$. Consequently $B \neq 1$ with $[G,B]\leq N\leq N_G(U)$
proving the theorem.
\end{proof}

Under the assumption that the action of $A$ on $G$ is good, the above
theorem would imply that $G=N_G(W)C_G(B)$. The following example shows that
we can not expect this equality without the assumption of ``goodness" of the
action.

\vspace{2mm}

\begin{example}
Let $\left\langle \sigma\right\rangle\cong \mathbb{Z}_9$ and $\alpha \in Aut
\left\langle \sigma\right\rangle$ given by $\alpha(\sigma)=\sigma^{4}$. Then
the group $S=\left\langle \sigma, \alpha\right\rangle$ is extraspecial of
order $27$ and of exponent 9 where $Z(S)=\left\langle \sigma^{3}\right\rangle
$. Let $R$ be a group isomorphic to $\mathbb{Z}_7$. Then $S$ acts on $R$ so
that $\left\langle \sigma^{3},\alpha\right\rangle$ forms the kernel of this
action. Put $G=R\left\langle \sigma\right\rangle$ and $N=R Z(S)$. Now $%
F=G/Z(S)$ is a Frobenius group of order $21$, acted on by $A=\left\langle
\alpha\right\rangle$. Here, $[G,A] \leq Z(S)$. $F$ has an irreducible
faithful character $\chi$ with $\chi(1)=3$. Considering $\chi$ as an
irreducible character of $G$ with kernel $Z(S)$, we see that  $%
\chi_{_{N}}=\theta_1 + \theta_2 + \theta_3 $, a sum of distinct irreducible
characters $\theta_1 , \theta_2 , \theta_3 $, each of which is fixed by $%
\alpha$. Let $V$ be a complex $GA$-module and $W_1$ be a complex $N$-module
affording $\chi$ and $\theta_1$, respectively. It is obvious that $G \neq
N_G(W_1) C_G(A)=N$ as $N_G(W_1)=N=C_G(A)$. Note also that $C_V(A)$ is $GA$%
-invariant as $[G,A]\leq Ker(G$ on $V)$. Therefore $C_V(A)=0 $ as $V$ is an
irreducible $GA$-module.
\end{example}

\begin{theorem}
Let $G$ be a solvable group on which a nilpotent group $A$ acts with regular
orbits. Suppose that the action of $A$ on $G$ is good. Let V be a complex $GA
$-module such that $V_G$ is homogeneous, $C_V(A)=0$ but $C_V(A_0)\ne 0$ for
every proper subgroup $A_0$ of $A.$ Let $N\lhd GA$ such that $[Z(N),A]$ acts
nontrivially on $V$ and let $W$ be a homogeneous component of $V_{_{N}}$.
Then $W$ is $A$-invariant and $[G,A]\leq N_G(W)$, that is, $G=N_G(W)C_G(A).$
\end{theorem}

\begin{proof}
Let $M$ be a normal subgroup of $G$ containing $N$ such that $G/M$ is a $GA$%
-chief factor of $G$. By Theorem 3.1 there exists a homogeneous component $U$
of $V_{_{M}}$ and a subgroup $B$ of $A$ such that $B\leq N_A(U)$, $C_V(B)=0$
and $[G,B]\leq N_G(U)$. By hypothesis we have $B=A$ and hence $[G,A]\leq
N_G(U)$.

As $M\leq N_G(U)\leq G$ we get either $M=N_G(U)$ and $N_G(U)=G$. Suppose
first that $N_G(U)=G$. Then $U=V$, that is $V_{_{M}}$ is homogeneous. Let $W$
be a homogeneous component of $V_{_{N}}$. By induction, applied to the
action of $MA$ on $V$, we get $[M,A]\leq N_G(W).$ Since the action is good,
this yields $M= N_G(W)C_M(A).$ Let now $T$ be a transversal for $N_G(W)$ in $%
M$. Clearly we may assume that $T\subseteq C_G(A).$ Now $V_{_{N}}=%
\bigoplus_{t\in T}W^t.$ Therefore for each $t\in T$, $W^t$ is $A$-invariant
whence 
\begin{equation*}
[Z(N),A]\leq \bigcap_{t\in T} Ker(W^t)=Ker(V).
\end{equation*}
This contradiction shows that $V_{_{M}}$ is not homogeneous. Thus we have $%
M=N_G(U)$ and hence $[G,A]\leq M,$ that is $G=MC_G(A)$ as the action is
good. Let now $S$ be a transversal for $M$ in $G.$ We may assume that $%
S\subseteq C_G(A)$ and hence $V_{_{M}}=\bigoplus_{s\in S}U^s.$ It also
follows that $U^s$ is $A$-invariant for all $s\in S$ as $U$ is $A$%
-invariant. If $[Z(N),A]$ is trivial on $U$ then it is trivial on $U^s$ for
all $s\in S$ as $S\subseteq C_G(A)$ and hence on $V$. Therefore $[Z(N),A]$
is nontrivial on $U$. Note that if $C_U(A_0)=0$ for some proper subgroup $A_0
$ of $A$ then $C_{U^s}(A_0)=0$ for all $s\in S$ and so $C_V(A_0)=0$, which
is not the case. Let $W_1$ be a homogeneous component of $U_{_{N}}$ such
that $W_1\subseteq W.$ By applying induction to the action of $MA$ on $U$ we
observe that $W_1$ is $A$-invariant and that $[M,A]\leq N_M(W_1)$. It
follows that $W$ is $A$-invariant and 
\begin{equation*}
[G,A]=[M,A]\leq N_G(W_1)\leq N_G(W).
\end{equation*}
This completes the proof.
\end{proof}

\section{fixed point free good action}

\begin{remark}
As an immediate consequence of Theorem 2.10 we observe that if $A$ is a
nilpotent group acting fixed point freely on the solvable group $G$ and the
action is good, then $h(G)\leq 4\ell(A)$. The next theorem improves this
bound.
\end{remark}

\begin{theorem}
Let $G$ be a solvable group and $A$ be a nilpotent group acting fixed point
freely on $G$. If the action is good, then $h(G)\leq 2\ell(A)$.
\end{theorem}

\begin{proof}
Suppose that $|A|=\prod_{i=1}^{k}{p_i}^{r_k}$. We use induction on $k$. If $%
k=1$, then $(|G|, |A|) = 1$ and the result is well-known by Corollary 3.2 in 
\cite{Tur}. Suppose that $p_j\notin \pi(G)$ for some $j\in \{1, 2,\ldots ,
k\}.$ Then the Sylow $p_j$-subgroup $A_j$ of $A$ is a group of automorphisms of $G$ of coprime order. By
inductive hypothesis and Corollary 3.2 in \cite{Tur} we have 
\begin{equation*}
h(C_G(A_j ))\leq 2\ell(A/A_j ).
\end{equation*}
Then the same corollary in \cite{Tur} implies that 
\begin{equation*}
h(G)\leq h(C_G(A_j )) + 2\ell(Aj )\leq 2\ell(A/Aj ) + 2\ell(Aj ) = 2\ell(A).
\end{equation*}
Hence one can suppose that $\{p_1, p_2,\ldots , p_k\}\subseteq \pi(G)$. By
Proposition 2.5 the group $A_i$ centralizes $G/O_{{p_i}^{\prime}}(G)$ and we
have $h(O_{{p_i}^{\prime}}(G))\leq 2\ell(A)$ by the previous argument. Let $N
=\prod_{i=1}^k O_{{p_i}^{\prime}}(G)$. Then $h(N) = \max{\{h(O_{{p_i}%
^{\prime}}(G)) : i=1,\ldots ,k \}}\leq 2\ell(A)$. Moreover $G/N$ is
centralized by every $A_i$, and so it is centralized by $A$, that is $G = N.$
\end{proof}

\begin{remark}
It can be seen that Theorem A in \cite{EG} can be extended as follows: Let $A
$ be an abelian group of squarefree exponent coprime to $6$ acting fixed
point freely on a group $G$ whose Sylow $2$-subgroups are abelian. Then $%
h(G)\leq \ell(A)$. The following theorem is similar to this result in the
sense that the assumption $(|A|,6)=1$ can be replaced by the goodness of the
action.
\end{remark}

\begin{theorem}
Let $G$ and $A$ be solvable groups such that the action of $A$ on $G$ is
good. Suppose that $G$ has abelian Sylow $2$-subgroups and that $A$ is
abelian of squarefree exponent. If $C_G(A)=1$ then $h(G)\leq \ell(A)$.
\end{theorem}

\begin{proof}
We proceed by induction on $\ell=\ell(A)$. The claim is well known in case
where $\ell= 1$. Thus we may assume that $\ell >1$. Let $p\in \pi(A)$ and
let $\alpha$ be a $p$-element of $A$. Then, by the fundamental result on the
structure of abelian groups, there is $B\leq A$ such that $A = B\oplus
\langle \alpha \rangle $. Let $C = C_G(\alpha)$. By induction applied to the
action of $B$ on $C$ we deduce that $h(C)\leq \ell-1$. On the other hand
appealing to Satz $3$ in \cite{K} we get $h([G,\alpha])\leq \ell.$ Set $N
=\prod_{1\ne a\in A}[G,a]$. Clearly, $N$ is a normal subgroup of $G$ with $%
h(N)\leq \ell.$ Since $A$ acts trivially on $G/N$ we have $G = N.$
\end{proof}

The rest of this section is devoted to the proof of the main result of this
paper.

\begin{theorem}
Let a finite nilpotent group $A$ of odd order which is $C_q\wr C_q$-free
for every prime $q$ act on the finite group $G$. Suppose that
\begin{itemize}
\item[(a)] $C_G(A) = 1$;
\item[(b)] the action of $A$ on $G$ is good;
\item[(c)] every subgroup of $A$ acts on $G$ with regular orbits;
\item[(d)] there is a subgroup $B$ of $A$ such that $\bigcap_{a\in A} \left[ G,B \right]^a =1$.
\end{itemize} 
Then $h(G)\leq \ell(A:B)$ where $%
\ell(A:B)$ is the number of prime divisors of $|A:B|$ counted with
multiplicities.
\end{theorem}

\begin{proof} We proceed by induction on $%
|G|+|A|+\ell(A:B).$ Set $h=h(G).$ The group $G$ is solvable by \cite{Belyaev}%
. As in the proof of Theorem 2.9 we see the existence of a sequence of
sections $P_1, \ldots, P_h$ of $G$ with $P_i = S_i / T_i$ where $S_i$ and $%
T_i$ are $A$-invariant subgroups of $G$ satisfying conditions (1)-(8) of
Definition 2.7. It should be noted that we may assume that $T_h = 1$ and $%
S_h \leq F (G)$.

To simplify the notation we set $V=P_h$ and $P=S_{h-1}$. By induction we
have $G=VPS_{h-2}\ldots S_1.$ Then we may assume that $\Phi(V)=1=T_{h-1}=1$
by corresponding induction arguments. Set now $X=PS_{h-2}\ldots S_1.$ By (8)
of Definition 2.7 $V$ is an irreducible $XA$-module. We shall proceed in a
series of steps:\newline

\textit{(1) $A$ acts faithfully on $G$, $A_1=C_A(P)\leq B$ and $(|P|,|A:B|)=1
$.}

\begin{proof}
By induction applied to the action of $A/Ker(A$ on $G)$ on $G$ with respect
to the subgroup $BKer(A$ on $G)/Ker(A$ on $G)$ we get $h\leq \ell(A:BKer(A$
on $G))$ which yields that $Ker(A$ on $G)\leq B$. Therefore we may assume
that $Ker(A$ on $G)=1$.

We can observe that $A_1=C_A(P)$ centralizes all the subgroups $P,
S_{h-2},\ldots ,S_1$ due to good action: Firstly we have $%
[S_{h-2}/T_{h-2},A_1]=1$ by the three subgroups lemma, whence $%
[S_{h-2},A_1]=1$ by Proposition 2.2 (3). Repeating the same argument we get
the claim.

Clearly $A_1\lhd A$. If $A_1\not \leq B$, by induction applied to the action
of $A/A_1 $ on the group $PS_{h-2}\ldots S_1$ with respect to the subgroup $%
BA_1/A_1$ we have $h-1\leq \ell(A/A_1:BA_1/A_1)$, which is a contradiction.
Thus $A_1\leq B$ and hence $(|P|,|A:B|)=1$ because $A_p$ centralizes $P$ by
Proposition 2.5.
\end{proof}

\textit{(2) For any subgroup $C$ of $A$ containing $B$ properly we have $P={%
\langle [P,C]\rangle}^X.$ }

\begin{proof}
Set $P_0={\langle [P,C]\rangle}^X $, and $X_0=S_{h-2}\ldots S_1.$ Suppose
that $P_0\ne P.$ Note that $P_0\lhd XC,$ and set $K=C_{X_0}(P/P_0).$ Then $%
P_0K\lhd PX_0C\lhd XC$. Since $[P,C]\leq P_0$ we have $[X_0,C]\leq K$ by the
three subgroups lemma. Then $[X,C]\leq P_0K$. Notice that $P_0S_{h-2}$ is
normalized by $P_0K$. If $P\leq P_0K$ then $P$ normalizes $P_0S_{h-2}$ and
so 
\begin{equation*}
P=[P,S_{h-2}]\leq [P,P_0S_{h-2}]\leq P_0S_{h-2}\cap P=P_0,
\end{equation*}
which is impossible. Thus we have $P\not \leq P_0K$ and so $P\not \leq
\bigcap_{a\in A}[X,C]^a$. This forces that $P\cap \bigcap_{a\in
A}[X,C]^a\leq \Phi(P)$ by condition (8) of Definition 2.7. Set $%
Y=\bigcap_{a\in A}[X,C]^a$ and $\bar{X}=X/Y.$ An induction argument applied
to the action of $A$ on $\bar{X}$ with respect to $C$ yields that 
\begin{equation*}
h(\bar{X})=h-1\leq \ell(A:C)=\ell(A:B)-\ell(C:B)
\end{equation*}
whence $h\leq \ell(A:B).$ This completes the proof of step \textit{(2)}.
\end{proof}

\textit{(3) There exists $B_1\leq A$ such that $B\lhd B_1$; and an
irreducible complex $XB_1$-submodule $M$ such that $M_{_{X}}$ is
homogeneous, $P\not \leq Ker (X$ on $M), [X,B]\leq Ker (X$ on $M), C_M(B)=M$
and $C_M(B_1)=0.$}

\begin{proof}
Clearly $V\not \leq [G,B]$ as $\bigcap_{a\in A}[G,B]^a=1.$ Note that $V_{XB}$
is completely reducible as $XB\lhd \lhd XA$. Let $W$ be an irreducible $XB$%
-submodule of $V$ such that $W\cap [G,B]=1$. Then $[W,B]=1$ and so $[X,B]\leq Ker(X$ on $W)$ by the three
subgroups lemma. Therefore $W_{_{X}}$ is homogeneous. Let $U$ be the $X$%
-homogeneous component of $V$ containing $W_{_{X}}$. Then $C_U(B)\ne 0,
P\not \leq Ker(X$ on $U)$, and $[X,B]\leq Ker(X$ on $U)$. Set now $B_0=N_A(U)
$. Then $U$ is an irreducible $XB_0$-module, and $B$ is properly contained
in $B_0$ as $C_U(B_0)=0.$

Let $\bar{k}$ be the algebraic closure of $k=\mathbb{F}_{p_{h}}$. Let $I$ be
an irreducible submodule of $U\otimes _k \bar{k}.$ By the Fong-Swan theorem
we may take an irreducible complex $XB_0$-module $M_0$ such that $Ker(X$ on $%
M_0)=Ker(X$ on $U)$ and $M_0$ gives $I$ when reduced modulo $p_h$. Thus $%
C_{M_0}(B)\ne 0,$ $P\not \leq Ker(X$ on $M_0)$, $[X,B]\leq Ker(X$ on $M_0)$
and $C_{M_0}(B_0)=0.$ Observe that $B$ normalizes each $X$-homogeneous
component of $M_0$ as $[X,B]\leq Ker(X$ on $M_0).$ Let now $M$ be an $X$%
-homogeneous component of $M_0$ such that $C_{M}(B)\ne 0.$ Set $%
B_1=N_{B_0}(M).$ Then we have $C_{M}(B_1)=0$ as $C_{M_0}(B_0)=0.$

Suppose that $B$ is not normal in $B_1,$ and let $C=\langle B^{B_1} \rangle.$
Then $[X,C]\leq Ker(X$ on $M)$. On the other hand, by \textit{(2)} we have $%
P={\langle [P,C]\rangle}^X.$ This forces that $P\leq [X,C]\lhd Ker(X$ on $M)$%
, which is not the case. Thus $B\lhd B_1$ whence $C_{M}(B)$ is a nonzero $%
XB_1$-submodule of $M,$ and so $[M,B]=0.$ Then $M$ is an irreducible $%
\mathbb{C}XB_1$-module such that $M_{_{X}}$ is homogeneous, $P\not \leq Ker(P
$ on $M)$, $[X,B]B \leq Ker(XB$ on $M)$, and $C_M(B_1)=0.$
\end{proof}

\textit{(4) Theorem follows.}

\begin{proof}
We consider the set of all pairs $(M_{\alpha}, C_{\alpha})$ such that $B\leq
C_{\alpha}\leq B_1 $, $M_{\alpha}$ is an irreducible $XC_{\alpha}$-submodule
of $M_{_{XC_{\alpha}}}$ and $C_{M_{\alpha}}(C_{\alpha})=0$. Choosing $(M_1,C)
$ with $|C|$ minimum. Then $C_{M_1}(C_0)\ne 0$ for every $B\leq C_0<C$, $%
(M_1)_{_{X}}$ is homogeneous and $Ker(X$ on $M_1)=Ker(X$ on $M).$

Set now $\bar{X}=X/Ker(P$ on $M)$. We can observe that $[Z(\bar{P}),C]=1$:
Otherwise, it follows by Theorem 3.3 that for any $\bar{P}$-homogeneous
component $U$ of $(M_1)_{_{\bar{P}}}$, the module $U$ is $C$-invariant and $%
\bar{X}=N_{\bar{X}}(U)C_{\bar{X}}(C)$. Then $C_{\bar{X}}(C)$ acts
transitively on the set of all $\bar{P}$-homogeneous components of $M_1$.
Clearly we have $[Z(\bar{P}),C]\leq Ker(\bar{P}$ on $U)$ and hence $[Z(\bar{P%
}),C]=1$, as claimed. Thus if $\bar{P}$ is abelian, then $[P,C]\leq Ker(\bar{%
P}$ on $M)$ and hence $P={\langle [P,C]\rangle}^X \leq Ker(\bar{P}$ on $M)$
by \textit{(2)}, which is not the case. Therefore $\bar{P}$ is nonabelian.

Let now $U$ be a homogeneous component of $(M_1)_{_{\Phi(\bar{P})}}$. Notice
that $\Phi(\bar{P})\leq Z(\bar{P})$ by (5) of Definition 2.7 and so $[\Phi(%
\bar{P}),C]=1$. Then $U$ is $C$-invariant. Set $\widehat{\bar{P}}=\bar{P}%
/Ker({\bar{P}}$ on $U)$. Now $\Phi(\widehat{\bar{P}})=\widehat{\Phi(\bar{P})}
$ is cyclic of prime order $p.$ Since $[Z(\bar{P}),C]=1$ we get $[X,C]\leq
C_X(Z(\bar{P}))$ by the three subgroups lemma. Now clearly we have $%
[X,C]\leq N_X(U).$ That is $X=N_X(U)C_X(C)$ as the action is good and so $%
C_X(C)$ acts transitively on the set of all homogeneous components of $%
(M_1)_{_{\Phi(\bar{P})}}$. Hence $M_1=\bigoplus_{t\in T}U^t$ where $T$ is a
transversal for $N_X(U)$ in $X$ contained in $C_X(C).$ Notice that $N_{\bar{X%
}C}(U)=N_{\bar{X}}(U)C.$ Set $X_1=C_X(\Phi(\bar{P}))$. Now $C_{XC}(\Phi(\bar{%
P}))=X_1C\lhd XC$ and we have $[X,C]\leq X_1$ by the three subgroups lemma.
Then $X=X_1C_X(C).$ Clearly we have $PS_{h-2}\leq X_1\leq N_X(U)$ and $%
X_1C\lhd XC\lhd \lhd XA.$ Recall that $P/\Phi(P)$ is an irreducible $XA$%
-module and hence $P/\Phi(P)$ is completely reducible as an $X_1C$-module.
Note that $\widehat{\bar{P}}/\Phi(\widehat{\bar{P}})\cong P/\Phi(P)C_P(U).$
As $P/\Phi(P)$ is completely reducible we see that so is $P/\Phi(P)C_P(U)$.
Hence $\widehat{\bar{P}}/\Phi(\widehat{\bar{P}})$ is also completely
reducible.

Since $\widehat{\Phi(\bar{P})}\leq \widehat{Z(\bar{P})}$, there
is an $X_1C$-invariant subgroup $E$ containing $\widehat{\Phi(\bar{P})}$ so
that 
\begin{equation*}
\widehat{\bar{P}}/\widehat{\Phi(\bar{P})}=\widehat{Z(\bar{P})}/\widehat{\Phi(%
\bar{P})}\oplus E/\widehat{\Phi(\bar{P})}.
\end{equation*}
Then $\widehat{\bar{P}}=\widehat{Z(\bar{P})}E$ and hence $\widehat{Z(\bar{P})%
}\cap E=Z(E).$ Clearly we have $({\widehat{\bar{P}}})^{\prime}=\widehat{\Phi(%
\bar{P})}\leq Z(E).$ Also, 
\begin{equation*}
E/\widehat{\Phi(\bar{P})}\cap \widehat{Z(\bar{P})}/\widehat{\Phi(\bar{P})}=1
\end{equation*}
and hence $Z(E)\leq \widehat{\Phi(\bar{P})}$. Thus we have $Z(E)=\widehat{%
\Phi(\bar{P})}=({\widehat{\bar{P}}})^{\prime}.$ As $E\unlhd \widehat{\bar{P}}$
we get $\Phi(E)\leq \widehat{\Phi(\bar{P})}=Z(E).$ It follows that $%
Z(E)=E^{\prime}=\Phi(E)=\widehat{\Phi(\bar{P})}$ is cyclic of prime order
and hence $E$ is extraspecial. Now $[Z(\bar{P}),C]=1$ gives $[\widehat{Z(%
\bar{P})},C]=1.$ Thus $[Z(E),C]=1.$

Next we observe that $C_{C}(E)\leq B$: Otherwise there is a nonidentity
element $b$ in $C\setminus B$ such that $[\widehat{\bar{P}},b]=1$ and hence $%
[\bar{P},b]\leq Ker(\bar{P}$ on $U)$. Since $X=X_1C_X(C)\leq N_X(U)C_X(C)$
we get $[\bar{P},b]\leq Ker(\bar{P}$ on $M)$. Set $C_1=B\langle b\rangle.$
Recall that $[P,B]\leq Ker(P$ on $M)$ by \textit{(3)}. Then, by \textit{(2)}%
, we have $P={\langle [P,C_1]\rangle}^X\leq Ker(P$ on $M)$ which is not the
case. Therefore $C_{C}(E)\leq B$ as claimed.

Notice that $p$ divides $|B/C_C(E)|$ if and only if $B_p\not \leq C_C(E)$
which is impossible by Proposition 2.5 due to good action. This means by 
\textit{(1)} that $p$ is coprime to $|C/C_C(E)|$. Note also that $B$ and
hence $C_C(E)$ acts trivially on $U$ by \textit{(3)}. We apply now Lemma 2.1
in \cite{Esp} to the action of the semidirect product $E(C/C_C(E))$ on the
module $U$ and see that $C_U(C/C_C(E))\ne 0$. This final contradiction
completes the proof.
\end{proof}
\end{proof}

\end{document}